\documentclass[a4paper,12pt]{article}

\usepackage[latin1]{inputenc}
\usepackage[T1]{fontenc}
\usepackage[english]{babel}

\usepackage{mathrsfs}
\usepackage{amscd}
\usepackage{amsfonts}
\usepackage{amsmath}
\usepackage{amssymb}
\usepackage{amstext}
\usepackage{amsthm}
\usepackage{amsbsy}

\usepackage{xspace}
\usepackage[all]{xy}
\usepackage{graphicx}
\usepackage{url}
\usepackage{latexsym}

\usepackage{graphicx} 

\usepackage{booktabs} 
\usepackage{array} 
\usepackage{paralist} 
\usepackage{verbatim} 
\usepackage{subfig} 

\usepackage{fancyhdr} 
\usepackage{sectsty}
\allsectionsfont{\sffamily\mdseries\upshape} 

\usepackage[nottoc,notlof,notlot]{tocbibind} 
\usepackage[titles,subfigure]{tocloft} 

\usepackage[textwidth=100pt,textsize=footnotesize,bordercolor=white,color=blue!30]{todonotes}
\usepackage{hyperref} 

\makeatletter
\newcommand*{\rom}[1]{\expandafter\@slowromancap\romannumeral #1@}
\makeatother

\theoremstyle{definition}

\newtheorem{fact}{fact}

\newtheorem{thm}[fact]{Theorem}
\newtheorem{lemma}[fact]{Lemma}
\newtheorem{prop}[fact]{Proposition}
\newtheorem{corollary}[fact]{Corollary}
\newtheorem{defini}[fact]{Definition}

\title{Randomness and Degree Theory for Infinite Time Register Machines}
\author{Merlin Carl}
\date{}

\begin{document}

\maketitle


\begin{abstract}
A concept of randomness for infinite time register machines (ITRMs) is defined and studied. In particular, we show that for this notion of randomness,
 computability from mutually random reals implies computability and that an analogue of van Lambalgen's theorem holds. This is then applied to obtain results on the structure of ITRM-degrees.
Finally, we consider autoreducibility for ITRMs and show that randomness implies non-autoreducibility. This is an expanded and amended version of \cite{CiE}.
\end{abstract}

\section{Introduction}

Martin-L\"of-randomness (ML-randomness, see e.g. \cite{DoHi}) provides an intuitive and conceptually stable clarification of the informal notion of a random sequence over a finite alphabet.
The guiding idea of ML-randomness is that a sequence of $0$ and $1$ is random if and only if it has no special properties, where a special property should be a small (e.g. measure $0$) set of reals
that is in some way accessible to a Turing machine.

Since its introduction, several variants of this general approach to defining randomness have been considered; a recent example is the work of Hjorth and Nies on $\Pi_{1}^{1}$-randomness and a $\Pi_{1}^{1}$ version of ML-randomness, 
which led to interesting connections with descriptive set theory (\cite{HN}).

We are interested in obtaining a similar notion based on machine models of transfinite computations.  In this paper, we will exemplarily consider infinite time register machines.
Infinite Time Register Machines (ITRMs), introduced in \cite{ITRM} and further studied in \cite{ITRM2}, work similar to the classical unlimited register machines (URMs) described in \cite{Cu}. 
In particular, they use finitely many registers each of which can store a single natural number. The difference is that ITRMs use transfinite ordinal running time: The state of an ITRM
at a successor ordinal is obtained as for URMs. At limit times, the program line is the inferior limit of the earlier program lines and there is a similar limit rule for the register contents. 
If the inferior limit of the earlier register contents is infinite, the register is reset to $0$.

Classical Turing machines, due to the finiteness of their running time, have the handicap that the only decidable null set of reals is the empty set: If a real $x$ is accepted by a classical Turing machine $M$
within $n$ steps, then $M$ will also accept every $y$ agreeing with $x$ on the first $n$ bits. In the definition of ML-randomness,
this difficulty is overcome by merely demanding the set $X$ in question to be effectively approximated by a recursively enumerable sequence of sets of intervals with controlled convergence behaviour.
For models of transfinite computations, this trick is unnecessary:
The decidable sets of reals form a rich class (including all ML-tests and, by \cite{ITRM}, all $\Pi_{1}^{1}$-sets). This is still a plausible notion of randomness, since elements of an ITRM-decidable meager set
can still be reasonably said to have a special property. In fact, some quite natural properties like coding a well-ordering can be treated very conveniently with our approach.
Hence, we define:

\begin{defini}{\label{deci}}
 $X\subseteq\mathfrak{P}(\omega)$ is called ITRM-decidable if and only if there is an ITRM-program $P$ such that $P^{x}\downarrow=1$ if and only if $x\in X$ and $P^{x}\downarrow=0$, otherwise. In this case we say that $P$ decides $X$.
$P$ is called deciding if and only if there is some $X$ such that $P$ decides $X$. We say that $X$ is decided by $P$ in the oracle $y$ if and only if $X=\{x\mid P^{x\oplus y}\downarrow=1\}$ and $\mathfrak{P}(\omega)-X=\{x\mid P^{x\oplus y}\downarrow=0\}$.
In this case, we also say that $P^{y}$ decides $X$.
The other notions relativize in the obvious way.
\end{defini}

\begin{defini}{\label{randomdef}} Recall that a set $X\subseteq\mathfrak{P}(\omega)$ is meager if and only if it is a countable union of nowhere dense sets.
$X\subseteq\mathfrak{P}(\omega)$ is an ITRM-test if and only if $X$ is ITRM-decidable and meager. $x\subseteq\omega$ is ITRM-c-random (where $c$ stands for `category') if and only if there is no ITRM-test $X$ such that $x\in X$. 
$x\subseteq\omega$ is ITRM-random if and only if there is no ITRM-decidable set $X\ni x$ of Lebesgue measure $0$. If there is no such set decidable in the oracle $y\subseteq\omega$,
then $x$ is called ITRM-c-random/ITRM-random relative to $y$.
\end{defini}

\noindent
\textbf{Remark}: Expect for the section on autoreducibility, we will mostly be concerned with ITRM-c-randomness. 
This obviously deviates from other definitions of randomness in that we use meager sets rather than null sets as our underlying notion of `small'. The reason is simply
that this variant turned out to be much more convenient to handle for technical reasons. The use of category rather than measure gives this definition a closer resemblance 
to what is, in the classical setting, refered to as genericity (see e.g. section 2.24 of \cite{DoHi}).
We still decided to use the term `ITRM-c-randomness' to avoid confusion with the frequently used concept of Cohen genericity, hence reserving the term `ITRM-random'
for reals that do not lie in any ITRM-decidable null set. 
We are pursuing the notion of ITRM-randomness in ongoing work.
In contrast to strong $\Pi_{1}^{1}$-randomness (\cite{HN}, \cite{Sa}), it will be shown below that there is no universal ITRM-test.

We will now summarize some key notions and results on ITRMs that will be used in the paper.

\begin{defini}
For $P$ a program, $x,y\in\mathfrak{P}(\omega)$, 
$P^{x}\downarrow=y$ means that the program $P$, when run with oracle $x$, halts on every input $i\in\omega$ and outputs $1$ if and only if $i\in y$ and $0$, otherwise.
$x\subseteq\omega$ is ITRM-computable in the oracle $y\subseteq\omega$ if and only if there is an ITRM-program $P$ such that $P^{y}\downarrow=x$, in which case we occasionally write $x\leq_{\text{ITRM}}y$. If $y$ can be taken to be
$\emptyset$, $x$ is ITRM-computable. We denote the set of ITRM-computable reals by COMP.
\end{defini}

\noindent
\textbf{Remark}: We occasionally drop the ITRM-prefix as notions like `computable' always refer to ITRMs in this paper.

\begin{thm}{\label{relITRM}}
 Let $x,y\subseteq\omega$. Then $x$ is ITRM-computable in the oracle $y$ if and only if $x\in L_{\omega_{\omega}^{\text{CK},y}}[y]$, where $\omega_{i}^{\text{CK},y}$ denotes the $i$th $y$-admissible ordinal.
\end{thm}
\begin{proof}
This is a straightforward relativization of Theorem $5$ of \cite{Ko2}, due to P. Koepke.
\end{proof}

\begin{thm}{\label{hp}}
Let $\mathbb{P}_{n}$ denote the set of ITRM-programs using at most $n$ registers, and let $(P_{i,n}|i\in\omega)$ enumerate $\mathbb{P}_{n}$ in some natural way. 
Then the bounded halting problem $H_{n}^{x}:=\{i\in\omega|P_{i,n}^{x}\downarrow\}$ is computable uniformly in the oracle $x$ by an ITRM-program (using more than $n$ registers).

Furthermore, if $P\in\mathbb{P}_{n}$ and $P^{x}\downarrow$, then $P^{x}$ halts in less than $\omega_{n+1}^{\text{CK},x}$ many steps. 
Consequently, if $P$ is a halting ITRM-program, then $P^{x}$ stops in less than $\omega_{\omega}^{\text{CK},x}$ many steps.
\end{thm}
\begin{proof}
 The corresponding results from \cite{ITRM} (Theorem $4$) and \cite{Ko2} (Theorem $9$) easily relativize. 
\end{proof}

\begin{defini}{\label{jump}}
 For $x\subseteq\omega$, $x^{\prime}_{\text{ITRM}}$ denotes the set of $i\in\omega$ such that $P_{i}^{x}\downarrow$. $x^{\prime}_{\text{ITRM}}$ is called the ITRM-jump of $x$.
We furthermore define the first $\omega$ many iterations of the ITRM-jump of $x$ by $x^{(0)}=x$, $x^{(i+1)}_{\text{ITRM}}=(x^{(i)}_{\text{ITRM}})^{\prime}$.
\end{defini}

We will freely use the following standard propositions:

\begin{prop}
 Let $X\subseteq[0,1]\times[0,1]$ and $\tilde{X}:=\{x\oplus y\mid (x,y)\in X\}$. Then $X$ is meager/comeager/non-meager if and only if $\tilde{X}$ is meager/comeager/non-meager.
\end{prop}

\begin{prop}
 If $X\subseteq[0,1]$ is meager, then so are $X\oplus[0,1]$ and $[0,1]\oplus X$.
\end{prop}

Most of our notation is standard. By a real, we mean an element of $^{\omega}2$.
$L_{\alpha}[x]$ denotes the $\alpha$th level of G\"odel's constructible hierarchy relativized to $x$. For $a,b\subseteq\omega$, $a\oplus b$ denotes $\{p(i,j)\mid i\in a\wedge j\in b\}$, where
$p:\omega\times\omega\rightarrow\omega$ is Cantor's pairing function. During the paper, we will frequently code countable $\in$-structures by reals. The idea here is the following: 
Let $X$ be transitive and countable, and let $f:\omega\leftrightarrow X$ be a bijection. Then $(X,\in)$ is coded by the real $r_{X,f}:=\{p(i,j)|i\in\omega\wedge j\in\omega\wedge f(i)\in f(j)\}\subseteq\omega$, and
it is easy to re-obtain $(X,\in)$ from $r_{X,f}$. In general, a real $x$ is a code for the transitive $\in$-structure $(X,\in)$ if and only if there is a bijection $f:\omega\rightarrow X$ such that
$x=r_{X,f}$. If $(X,\in)$ is constructible and $X$ is transitive and countable in $L$, then $L$ contains a $<_{L}$-minimal code for $(X,\in)$; this code will be called the canonical code for $(X,\in)$ and
denoted by $\text{cc}(X)$.

\section{Computability from random oracles}

In this section, we consider the question which reals can be computed by an ITRM with an ITRM-c-random oracle. We start by recalling
 the following theorem from \cite{CaSc}. The intuition behind it is that, given a certain non-ITRM-computable real $x$, one has no chance of computing $x$ from
some randomly chosen real $y$.

\begin{thm}{\label{ManyOracles}}
Let $x$ be a real, $Y$ be a set of reals such that $x$ is ITRM-computable from every $y\in Y$.
Then, if $Y$ has positive Lebesgue measure or is Borel and non-meager, $x$ is ITRM-computable.
\end{thm}

\begin{corollary}{\label{Admpreserv}}
 Let $x$ be ITRM-c-random. Then, for all $i\in\omega$, $\omega_{i}^{\text{CK},x}=\omega_{i}^{\text{CK}}$.
\end{corollary}
\begin{proof}
Lemma $46$ of \cite{CaSc} shows that $\omega_{i}^{\text{CK},x}=\omega_{i}^{\text{CK}}$ for all $i\in\omega$ whenever $x$ is Cohen-generic over $L_{\omega_{\omega}^{\text{CK}}}$ (see e.g. \cite{CaSc} or \cite{Ma})
 and that the set of Cohen-generics over $L_{\omega_{\omega}^{\text{CK}}}$
is comeager. Hence $\{x|\omega_{i}^{\text{CK},x}>\omega_{i}^{\text{CK}}\}$ is meager.
For each program $P$, the set of reals $x$ such that $\omega_{i}^{\text{CK},x}>\omega_{i}^{\text{CK}}$ and $P^{x}$ computes a code for $\omega_{i}^{\text{CK},x}$ is decidable using the techniques 
developed in \cite{Ca} and \cite{Ca2}. (The idea is to uniformly in the oracle $x$ compute a real $c$ coding $L_{\omega_{i+1}^{\text{CK},x}}[x]$ in which the natural numbers $m$ and $n$ coding
$\omega_{i}^{\text{CK}}$ and $\omega_{i}^{\text{CK},x}$ can be identified in the oracle $x$, and then to check - using a halting problem solver
for $P$, see Theorem \ref{hp} - whether $P^{x}$ computes a well-ordering of the same order type as the element of $L_{\omega_{i+1}^{\text{CK},x}}[x]$ coded by $n$ and 
finally whether the element coded by $m$ is an element of that coded by $n$.) 
Hence, if $x$ is ITRM-c-random, then there can be no ITRM-program $P$ computing such a code in the oracle $x$. But a code for $\omega_{i}^{\text{CK},x}$ is
ITRM-computable in the oracle $x$ for every real $x$ and every $i\in\omega$. Hence, we must have $\omega_{i}^{\text{CK},x}=\omega_{i}^{\text{CK}}$ for every $i\in\omega$, 
as desired.
\end{proof}

A notable conceptual difference between ITRM-c-randomness and Martin-L\"of-randomness is the absence of a universal test for the former:

\begin{thm}{\label{nouniversaltest}}
 There is no universal test for ITRM-c-randomness, i.e. the union of all ITRM-decidable meager sets is not ITRM-decidable.
\end{thm}
\begin{proof}
 Assume otherwise. Let $U$ be the union of all ITRM-decidable meager sets and let $P$ be an ITRM-program such that $P$ decides $U$. Clearly, $U$, as a countable union of
meager sets, is meager. Let $P$ have
$n$ registers. Let $M$ be the set of reals $x$ that are Cohen-generic over $L_{\omega_{n+1}^{\text{CK}}}$, but not over $L_{\omega_{n+2}^{\text{CK}}}$. As the set of
reals which are Cohen-generic over $L_{\omega_{n+2}^{\text{CK}}}$ is comeager, $M$ is a subset of a meager set and hence meager. It is also easy to see that
$M$ is ITRM-decidable: To decide $M$, compute a code for $L_{\omega_{n+2}^{\text{CK}}}$; then, given a real $y$ in the oracle, search through the dense
subsets of Cohen-forcing in $L_{\omega_{n+2}^{\text{CK}}}$ to see whether $x$ intersects all those contained in $L_{\omega_{n+1}^{\text{CK}}}$, but fails to intersect at least one dense subset contained in $L_{\omega_{n+2}^{\text{CK}}}$.

Now pick $x\in M$. As $x\in M$ and $M$ is ITRM-decidable and meager,
we have $x\in U$, so $P^{x}\downarrow=1$. By genericity of $x$ over $L_{\omega_{n+1}^{\text{CK}}}$ and the forcing theorem, there is a condition $p$ of Cohen-forcing such that
$p\subseteq x$ and $p\Vdash P^{x}\downarrow=1$. Consequently, we have $p\Vdash P^{y}\downarrow=1$ for every $L_{\omega_{n+1}^{\text{CK}}}$-generic real $y$ extending $p$.
By absoluteness of computations, it follows that $P^{y}\downarrow=1$ for each such $y$.
Since $P$ decides $U$, each such $y$ hence belongs to $U$. But the set of these $y$ is comeager in some interval and hence not meager.
Hence $U$ has a non-meager subset and his thus not meager, a contradiction.
\end{proof}

\begin{lemma}{\label{compandgen}}
Let $a\subseteq\omega$ and suppose that $z$ is Cohen-generic over $L_{\omega_{\omega}^{\text{CK},a}+1}[a]$. Then $a\leq_{\text{ITRM}}z$ if and only if $a$ is ITRM-computable.
Consequently (as the set $C_{a}:=\{z\subseteq\omega\mid z\text{ is Cohen-generic over }L_{\omega_{\omega}^{\text{CK},a}+1}[a]\}$ is comeager), the set $S_{a}:=\{z\subseteq\omega\mid a\leq_{\text{ITRM}}z\}$ is meager whenever $a$ is not ITRM-computable.
\end{lemma}
\begin{proof}
Assume that $z$ is Cohen-generic over $L_{\omega_{\omega}^{\text{CK},a}+1}[a]$ and $a\leq_{\text{ITRM}}z$. By the forcing theorem for provident sets (see e.g. Lemma $32$ of \cite{CaSc}), there is an ITRM-program $P$ and
a forcing condition $p$ such that $p\Vdash P^{\dot{G}}\downarrow=\check{a}$, where $\dot{G}$ is the canonical name for the generic filter and $\check{a}$ is the canonical name of $a$.
 Now, let $y$ and $z$ be mutually Cohen-generic over $L_{\omega_{\omega}^{\text{CK},a}+1}[a]$ both extending $p$.
Again by the forcing theorem and by absoluteness of computations, we must have $P^{x}\downarrow=a=P^{y}\downarrow$, so $a\in L_{\omega_{\omega}^{\text{CK},x}}[x]\cap L_{\omega_{\omega}^{\text{CK},y}}[y]$.
By Corollary \ref{Admpreserv}, $\omega_{\omega}^{\text{CK},x}=\omega_{\omega}^{\text{CK},y}=\omega_{\omega}^{\text{CK}}$. By Lemma $28$ of \cite{CaSc}, we have $L_{\alpha}[x]\cap L_{\alpha}[y]=L_{\alpha}$ whenever
$x$ and $y$ are mutually Cohen-generic over $L_{\alpha}$ and $\alpha$ is provident (see \cite{Ma}). Consequently, we have:
\begin{center} $a\in L_{\omega_{\omega}^{\text{CK},x}}[x]\cap L_{\omega_{\omega}^{\text{CK},y}}[y]=L_{\omega_{\omega}^{\text{CK}}}[x]\cap L_{\omega_{\omega}^{\text{CK}}}[y]
=L_{\omega_{\omega}^{\text{CK}}}$,\end{center}
 so $a$ is ITRM-computable.

The comeagerness of $C_{a}$ is standard (see e.g. Lemma $29$ of \cite{CaSc}). To see that $S_{a}$ is meager for non-ITRM-computable $a$, observe that the Cohen-generic reals over $L_{\omega_{\omega}^{\text{CK},a}+1}[a]$
form a comeager set of reals to none of which $a$ is reducible.
\end{proof}

\begin{defini}
Let $x,y\subseteq\omega$. 
If $x$ is ITRM-c-random relative to $y$ and $y$ is ITRM-c-random relative to $x$, we say that $x$ and $y$ are mutually ITRM-c-random.
\end{defini}

Intuitively, we should expect that mutually random reals have no non-trivial information in common. This is expressed by the following theorem:

\begin{thm}{\label{mutuallyrandom}}
 If $z$ is ITRM-computable from two mutually ITRM-c-random reals $x$ and $y$, then $z$ is ITRM-computable.
\end{thm}
\begin{proof}
Assume otherwise, and suppose that $z$, $x$ and $y$ constitute a counterexample. By assumption, $z$ is computable from $x$. Also, by Theorem \ref{hp}, let $P$ be a program such that 
$P^{a}(i)\downarrow$ for every $a\subseteq\omega$, $i\in\omega$ and such that $P$
computes $z$ in the oracle $y$. In the oracle $z$, the set $A_{z}:=\{a|\forall{i\in\omega}P^{a}(i)\downarrow=z(i)\}$ is decidable by simply computing $P^{a}(i)$ for all $i\in\omega$ and comparing the result to the $i$th bit of $z$.
Clearly, we have $A_{z}\subseteq\{a\mid z\leq_{\text{ITRM}}a\}$. Hence, by our Lemma \ref{compandgen} above, $A_{z}$ is meager as $z$ is not ITRM-computable by assumption. 
Since $A_{z}$ is decidable in the oracle $z$ and $z$ is computable from $x$, $A_{z}$ is also decidable in the oracle $x$. Now, $x$ and $y$ are 
mutually ITRM-c-random, so that $y\notin A_{z}$. But $P$ computes $z$ in the oracle $y$, so $y\in A_{z}$ by definition, a contradiction.
\end{proof}

While, naturally, there are non-computable reals that are reducible to a $c$-random real $x$ (such as $x$ itself), intuitively, it should not be possible to compute a non-computable real that is `unique' is some effective sense from a random real.
We approximate this intuition by taking `unique' to mean `ITRM-recognizable' (see \cite{ITRM2}, \cite{Ca} or \cite{Ca2} for more information on ITRM-recognizability).
 It turns out that, in accordance with this intuition, recognizables that are ITRM-computable from ITRM-c-random reals are already ITRM-computable. 

\begin{defini}{\label{recogdef}}
$x\subseteq\omega$ is ITRM-recognizable if and only if $\{x\}$ is ITRM-decidable. RECOG denotes the set of recognizable reals.
\end{defini}

\begin{thm}{\label{weakrandomreducibility}}
 Let $x\in$ RECOG and let $y$ be ITRM-c-random such that $x\leq_{\text{ITRM}}y$. Then $x$ is ITRM-computable.
\end{thm}
\begin{proof}
Let $x\in$ RECOG$\setminus$COMP be computable from $y$, say by program $P$, and let $Q$ be a program that recognizes $x$. 
The set $S:=\{z\mid P^{z}\downarrow=x\}$ is meager as in the proof of Theorem \ref{mutuallyrandom}. But $S$ is decidable:
Given a real $z$, use a halting-problem solver for $P$ (which exists uniformly in the oracle by Theorem \ref{hp}) to test whether $P^{z}(i)\downarrow$ for all $i\in\omega$; if not,
then $z\notin S$. Otherwise, use $Q$ to check whether the real computed by $P^{z}$ is equal to $x$. If not, then $z\notin S$, otherwise $z\in S$. As $P^{y}$ computes $x$,
it follows that $y\in S$, so that $y$ is an element of an ITRM-decidable meager set. Hence $y$ is not ITRM-c-random, a contradiction.
\end{proof}

\noindent
\textbf{Remark}: Let $(P_{i}|i\in\omega)$ be a natural enumeration of the ITRM-programs.
 Together with the fact that the halting number $h=\{i\in\omega\mid P_{i}\downarrow\}$ for ITRMs is recognizable (see \cite{Ca2}), this implies in particular that the halting problem for ITRMs is not 
ITRM-reducible to an ITRM-c-random real. In particular, the Kucera-Gacs theorem, which says that every real is reducible to a random real (see e.g. Theorem $8.3.2$ of \cite{DoHi}), does not hold in our setting.

\section{An analogue to van Lambalgen's theorem}

A crucial result of classical algorithmic randomness is van Lambalgen's theorem, which
 states that for reals $a$ and $b$, $a\oplus b$ is ML-random if and only if $a$ is ML-random and $b$ is ML-random relative to $a$.
In this section, we demonstrate an analogous result for ITRM-c-randomness. This will be a crucial ingredient in our
considerations on ITRM-degrees below.
%

\begin{lemma}
 Let $Q$ be a deciding ITRM-program using $n$ registers and $a\subseteq\omega$. Then $\{y|Q^{y\oplus a}\downarrow=1\}$ is meager if and only if
$Q^{x\oplus a}\downarrow=0$ for all $x\in L_{\omega_{n+1}^{\text{CK},a}+3}[a]$ that are Cohen-generic over $L_{\omega_{n+1}^{\text{CK},a}+1}[a]$.
\end{lemma}
\begin{proof}
By absoluteness of computations and the bound on ITRM-halting times (see Theorem \ref{hp}), 
$Q^{x\oplus a}\downarrow=0$ implies that $Q^{x\oplus a}\downarrow=0$ also holds in $L_{\omega_{n+1}^{\text{CK},a}}[a]$. As this is expressible by a $\Sigma_1$-formula, it must be forced by some condition $p$ by the 
forcing theorem over KP (see e.g. Theorem $10.10$ of \cite{Ma}).

Hence every Cohen-generic $y$ extending $p$ will satisfy $Q^{y\oplus a}\downarrow=0$. The set $C$ of reals Cohen-generic over $L_{\omega_{n+1}^{\text{CK},a}+1}[a]$ 
is comeager. Hence, if $Q^{x\oplus a}\downarrow=0$ for some $x\in C$, then $Q^{x\oplus a}\downarrow=0$ for a non-meager (in fact comeager in some interval) set $C^{\prime}$.
Now, for each condition $p$, $L_{\omega_{n+1}^{\text{CK},a}+3}[a]$ will contain a generic filter over $L_{\omega_{n+1}^{\text{CK},a}+1}[a]$ extending $p$ (as $L_{\omega_{n+1}^{\text{CK},a}+1}[a]$ is countable
in $L_{\omega_{n+1}^{\text{CK},a}+3}[a]$). 
Hence, if $Q^{x\oplus a}\downarrow=0$ for all $x\in C\cap L_{\omega_{n+1}^{\text{CK},a}+3}[a]$, then this holds for all elements of $C$ and the complement $\{y|Q^{y\oplus a}\downarrow=1\}$ is therefore meager.

If, on the other hand, $Q^{x\oplus a}\downarrow=1$ for some such $x$, then this already holds for all $x$ in some non-meager (in fact comeager in some interval) set $C^{\prime}$ by the same reasoning. 
\end{proof}

\begin{corollary}{\label{decidemeasure}}
 For a deciding ITRM-program $Q$ using $n$ registers, there exists an ITRM-program $P$ such that, for all $x,y\in\mathfrak{P}(\omega)$, $P^{x}\downarrow=1$ if and only if $\{y|Q^{x\oplus y}\downarrow=1\}$ is non-meager and
$P^{x}\downarrow=0$, otherwise.
\end{corollary}
\begin{proof}
 From $x$, compute, using sufficiently many extra registers, the $<_{L}$-minimal real code $c:=\text{cc}(L_{\omega_{n+1}^{\text{CK},x}+3}[x])$ for $L_{\omega_{n+1}^{\text{CK},x}+3}[x]$ in the oracle $x$. This can be done uniformly in $x$. 
Then, using $c$, one can use the recursive algorithm developed in
in section $6$ of \cite{ITRM} to evaluate statements in $L_{\omega_{n+1}^{\text{CK},x}+3}[x]$. Hence, we can search through $c$, identify all elements which code reals $y\in L_{\omega_{n+1}^{\text{CK},x}+3}[x]$ that are Cohen-generic
over $L_{\omega_{n+1}^{\text{CK},x}+1}$ and run $Q^{x\oplus y}$ for each of them to see whether $Q^{x\oplus y}\downarrow=1$; if yes, then we return $1$, otherwise, we return $0$.
\end{proof}

\begin{corollary}{\label{randomandcohengeneric}}
Let $x,y$ be real numbers. Then $x$ is ITRM-c-random in the oracle $y$ if and only if $x$ is Cohen-generic over $L_{\omega_{\omega}^{\text{CK},y}}[y]$.
\end{corollary}
\begin{proof}
Let $S$ denote the set of Cohen-generic reals over $L_{\omega_{\omega}^{\text{CK},y}}[y]$.
Then $x\in S$ if and only if $x\cap D\neq\emptyset$ for every dense subset $D\in L_{\omega_{\omega}^{\text{CK},y}}[y]$ of Cohen-forcing. 
Clearly, for every such $D$, $G_{D}:=\{y\mid y\cap D\neq\emptyset\}$ is comeager and ITRM-decidable in the oracle $y$ (and so its complement is ITRM-decidable in $y$ and meager), 
so every real number that is ITRM-c-random real relative to $y$ must be contained in every $G_{D}$ and hence also in $S$.

On the other hand, if $x\in S$ and $P^{x\oplus y}\downarrow=1$ for some ITRM-program $P$ that is deciding in the oracle $y$,
then there is some finite $p\subseteq x$ such that $P^{z\oplus y}\downarrow=1$ for every $p\subset z\in S$, so the set decided
by $P$ in $y$ is not meager. Hence $x$ is not an element of any meager set ITRM-decidable
in the oracle $y$, so $x$ is ITRM-c-random relative to $y$.
\end{proof}



We now give an ITRM-analogue of van Lambalgen's theorem and prove it following the strategy used in the classical setting, see e.g. \cite{DoHi}, Theorem $6.9.1$ and $6.9.2$.

\begin{thm}{\label{Dir1}}
 Assume that $a$ and $b$ are reals such that $a\oplus b$ is not ITRM-c-random. Then $a$ is not ITRM-c-random or $b$ is not ITRM-c-random relative to $a$.
\end{thm}
\begin{proof}
As $a\oplus b$ is not ITRM-c-random, let $X$ be an ITRM-decidable meager set of reals such that $a\oplus b\in X$. Suppose that $P$ is a program deciding $X$.

Let $Y:=\{x|\{y\mid x\oplus y\in X\}\text{ non-meager}\}$. By Corollary \ref{decidemeasure}, $Y$ is ITRM-decidable.

We claim that $Y$ is meager. First, $Y$ is provably $\Delta_{2}^{1}$ and hence has the Baire property (see e.g. Exercise $14.4$ of \cite{Ka}). Hence, by
the Kuratowski-Ulam-theorem (see e.g. \cite{Mo}, Exercise 5A.9), $Y$ is meager. Consequently, if $a\in Y$, then $a$ is not ITRM-c-random.

Now suppose that $a\notin Y$. This means that $\{y\mid a\oplus y\in X\}$ is meager. But $S:=\{y\mid a\oplus y\in X\}$ is easily seen to be ITRM-decidable in the oracle $a$ and
$b\in S$. Hence $b$ is not ITRM-c-random relative to $a$.
\end{proof}

\begin{thm}{\label{Dir2}}
 Assume that $a$ and $b$ are reals such that $a\oplus b$ is ITRM-c-random. Then $a$ is ITRM-c-random and $b$ is ITRM-c-random relative to $a$.
\end{thm}
\begin{proof}
 Assume first that $a$ is not ITRM-c-random, and let $X$ be an ITRM-decidable meager set with $a\in X$.
Then $X\oplus[0,1]$ is also meager and ITRM-decidable. As $a\in X$, we have $a\oplus b\in X\oplus[0,1]$, so
$a\oplus b$ is not ITRM-c-random, a contradiction.

Now suppose that $b$ is not ITRM-c-random relative to $a$, and let $X$ be a meager set of reals such that $b\in X$ and $X$ is ITRM-decidable in the oracle $a$.
Let $Q$ be an ITRM-program such that $Q^{a}$ decides $X$. Our goal is to define a deciding program $\tilde{Q}$ such that $\tilde{Q}^{a}$ still decides $X$, but
also $\{x|\tilde{Q}^{x}\downarrow=1\}$ is meager. This suffices, as then $\tilde{Q}^{a\oplus b}\downarrow=1$ and $\{x|\tilde{Q}^{x}\downarrow=1\}$ is ITRM-decidable, so that $a\oplus b$ cannot be ITRM-c-random.
$\tilde{Q}$ operates as follows: Given $x=y\oplus z$, check whether $\{w\mid Q^{y\oplus w}\downarrow=1\}$ is meager, using Corollary \ref{decidemeasure}. If that is the case,
carry out the computation of $Q^{x}$ and return the result. Otherwise, return $0$. This guarantees (since $X$ is meager) 
that $\{z\mid \tilde{Q}^{y\oplus z}\downarrow=1\}$ is meager and furthermore that $\tilde{Q}^{a\oplus z}\downarrow=1$ if and only if $Q^{a\oplus z}\downarrow=1$ if and only if
$z\in X$ for all reals $z$, so that $\{z|\tilde{Q}^{a\oplus z}\downarrow=1\}$ is just $X$, as desired.
\end{proof}

Combining Theorem \ref{Dir1} and \ref{Dir2} gives us the desired conclusion:

\begin{thm}{\label{vLamb}}
 Given reals $x$ and $y$, $x\oplus y$ is ITRM-c-random if and only if $x$ is ITRM-c-random and $y$ is ITRM-c-random relative to $x$.
In particular, if $x$ and $y$ are ITRM-c-random, then $x$ is ITRM-c-random relative to $y$ if and only if $y$ is ITRM-c-random relative to $x$.
\end{thm}

\textbf{Remark}: Combined with Corollary \ref{randomandcohengeneric}, this shows `computationally' that, over $L_{\omega_{\omega}^{\text{CK}}}$, $x\oplus y$ is Cohen-generic if and only if $x$ is Cohen-generic
over $L_{\omega_{\omega}^{\text{CK}}}$ and $y$ is Cohen-generic over $L_{\omega_{\omega}^{\text{CK}}}[x]$.


We note that a classical corollary to van Lambalgen's theorem continues to hold in our setting:

\begin{corollary}
 Let $x,y$ be ITRM-c-random. Then $x$ is ITRM-c-random relative to $y$ if and only if $y$ is ITRM-c-random relative to $x$.
\end{corollary}
\begin{proof}
 Assume that $y$ is ITRM-c-random relative to $x$. By assumption, $x$ is ITRM-c-random. By Theorem \ref{vLamb}, $x\oplus y$ is ITRM-c-random. Trivially, $y\oplus x$ is also ITRM-c-random.
Again by Theorem \ref{vLamb}, $x$ is ITRM-c-random relative to $y$. By symmetry, the corollary holds.
\end{proof}

Based on Theorem \ref{weakrandomreducibility}, one might conjecture that a real $x$ which is ITRM-computable from an ITRM-c-random real $y$ must be computable or itself ITRM-c-random.
This, however does not hold:


\begin{thm}{\label{cntblrnd}}
For every ITRM-c-random real $y$, there is a real $x$ such that $x$ is neither ITRM-c-random nor ITRM-computable and $x$ is ITRM-computable from $y$. 
However, if $x$ is ITRM-computable from
an ITRM-c-random real $y$ and $x$ is an element of an ITRM-decidable set $S$ such that $\text{card}(S)=\aleph_{0}$, then $x$ is computable.
\end{thm}
\begin{proof}
 Let $x:=0\oplus y$, then $x$ is clearly computable from $y$. However, $x$ is not computable as the non-computable $y$ is computable from $x$.
Furthermore, $x$ is contained in the ITRM-decidable meager set $\{0\}\oplus[0,1]$, so $x$ cannot be ITRM-c-random.

For the second statement, assume for a contradiction that $S$ is an ITRM-decidable countable set containing a non-computable real $x$ such that $x\leq_{\text{ITRM}}y$ for some ITRM-c-random real $y$.
Let $P$ be a program deciding $S$ and let $Q$ be a program such that $Q^{y}\downarrow=x$. By Theorem \ref{hp}, we may assume without loss of generality that $Q^{z}$ computes some real for every $z\subseteq\omega$.
By $\bar{S}$, we denote the set of $z\in S$ such that $M_{a}:=\{a|Q^{a}\downarrow=z\}$ is meager, $Q(z)$ denotes the real number computed by $Q^{z}$. Then, as $x$ is not ITRM-computable, we have $x\in\bar{S}$ by Theorem \ref{ManyOracles}. Furthermore
$\bar{S}$ is decidable: For let $R$ be a program such that $R^{a\oplus b}\downarrow=1$ if and only if $Q^{a}\downarrow=b$ and $R^{a\oplus b}\downarrow=0$ otherwise; so $R$ is deciding.
Now $a\in\bar{S}$ if and only if $a\in S$ and $\{b|R^{a\oplus b}\}$ is meager, hence, by Theorem \ref{decidemeasure}, $\bar{S}$ is decidable by some program $\bar{P}$.
Let $M:=\{z|\bar{P}^{Q(z)}\downarrow=1\}$ be the set of all reals $z$ such that $Q(z)\in\bar{S}$. Then $M$ is obviously ITRM-decidable.
As $Q(y)=x\in\bar{S}$, we have $y\in M$. As $y$ is ITRM-c-random, it follows that $M$ is not meager. But $M=\bigcup_{a\in S}M_{a}$; as
$S$ is countable, there must thus be some $a\in S$ such that $M_{a}$ is not meager, a contradiction.
\end{proof}

Note that the condition of being ITRM-computable from an ITRM-c-random real is necessary in the assumption of Theorem \ref{cntblrnd}:
It is not true that every element of a countable ITRM-decidable set is ITRM-computable; in fact, there are countable ITRM-decidable sets that do not contain any ITRM-computable element:

%

\begin{prop}
The set $X:=\{x\subseteq\omega:\exists\alpha\in\text{On}(\omega_{\omega}^{\text{CK}}<\alpha<\omega_{2\omega}^{\text{CK}}\wedge`x$ is the $<_{L}$-minimal real in $L_{\alpha+1}\setminus L_{\alpha}$ coding $L_{\alpha}$'$)\}$ 
is countable, ITRM-decidable and contains no ITRM-computable element.
\end{prop}
\begin{proof}
As $\omega_{2\omega}^{\text{CK}}$ is countable, the set of indices between $\omega_{\omega}^{\text{CK}}$ and $\omega_{2\omega}^{\text{CK}}$ is countable and hence so is the set of the corresponding codes for
$L$-levels. Also, by Theorem \ref{relITRM}, all ITRM-computable reals are contained in $L_{\omega_{\omega}^{\text{CK}}}$, so $X$ has no ITRM-computable element.

To see that $X$ is ITRM-decidable, we apply the techniques from section $6$ of \cite{ITRM2} to check, for a given oracle $x$, whether $x$ codes an $L$-level $L_{\alpha}$ and whether this $L_{\alpha}$ contains exactly one limit of admissible ordinals.
It is also shown in \cite{ITRM2} how to uniformly compute a code $c$ for $L_{\alpha+1}$ from $x$ when $x$ codes $L_{\alpha}$. 
Hence, it only remains to check whether $L_{\alpha+1}\setminus L_{\alpha}$ contains $x$ and whether $x$ is $<_{L}$-minimal in $L_{\alpha+1}$
with these properties, and the answers to both questions can be easily read off from $c$.
\end{proof}

\textbf{Remark}: We do not know whether a countable ITRM-decidable set $X$ can contain a non-ITRM-recognizable element. Our conjecture is that this is not possible and that in fact
any ITRM-decidable set with a non-ITRM-recognizable element must have a perfect subset (and hence have cardinality $2^{\aleph_{0}}$).
The next result, however, shows that every non-empty ITRM-decidable set (countable or not) contains \textbf{some} recognizable element. This can be seen
as a kind of basis theorem for ITRMs:

\begin{thm}{\label{ITRMbase}}
Every non-empty ITRM-decidable set of real numbers has a recognizable element.
\end{thm} 
\begin{proof}
 Let $X\neq\emptyset$ be ITRM-decidable, and let $P$ be a program that decides $X$. Let $\Gamma:=\{\omega_{\omega}^{\text{CK},x}:x\in X\}$. Since $X$ is non-empty, so is $\Gamma$; let $\mu$ be the minimal element of $\Gamma$ and pick
$x_{\gamma}\in X$ such that $\omega_{\omega}^{\text{CK},x_{\gamma}}=\mu$. By Theorem \ref{hp}, we thus have $P^{x_{\gamma}}\downarrow=1$ in less than $\mu$ many steps. By absoluteness of computations, it follows that
$L_{\mu}[x_{\gamma}]\models P^{x_{\gamma}}\downarrow=1$ and hence $L_{\mu}[x_{\gamma}]\models\exists{x}P^{x}\downarrow=1$. By the Jensen-Karp theorem (see \cite{JeKa}), the $\Sigma_1$-formula $\exists{x}P^{x}\downarrow=1$ is absolute between
$L_{\mu}[x_{\gamma}]$ and $L_{\mu}$; hence $L_{\mu}\models\exists{x}P^{x}\downarrow=1$. Let $\bar{x}\in L_{\mu}$ be the $<_{L}$-minimal element of $L_{\mu}$ such that $L_{\mu}\models P^{\bar{x}}\downarrow=1$; thus $\bar{x}\in X$
since $P^{\bar{x}}\downarrow=1$ by absoluteness of computations.

We claim that $\bar{x}$ is recognizable. To see this, first note that $\bar{x}\in L_{\omega_{\omega}^{\text{CK},\bar{x}}}$, so that there is $k\in\omega$ with $\bar{x}\in L_{\omega_{k}^{\text{CK},\bar{x}}}$.
Suppose that $P$ uses $n$ registers and let $m=\text{max}\{k,n+1\}$. Then the $\Sigma_1$-hull of $\{\bar{x}\}$ in $L_{\omega_{m}^{\text{CK},\bar{x}}}$ will be isomorphic (via the collapsing map) to $L_{\omega_{m}^{\text{CK},\bar{x}}}$,
and so $L_{\omega_{m}^{\text{CK},\bar{x}}+1}$ will contain a real number $r$ coding $L_{\omega_{m}^{\text{CK},\bar{x}}}$ by standard fine-structural arguments. Hence $r$ is ITRM-computable from $\bar{x}$; let $Q$ be an ITRM-program computing $r$ from $\bar{x}$.
To recognize $\bar{x}$, we now proceed as follows: Let $y$ be given in the oracle. First test whether $P^{y}\downarrow=1$. If not, then $y\neq\bar{x}$. Otherwise, test, 
using $Q$ and a halting problem solver for $Q$, whether $Q^{y}$ computes a real $r$ coding an admissible $L_{\alpha}$ containing $y$.
If not, then $y\neq\bar{x}$. If yes, then use $r$ to search for an element $z$ of $L_{\alpha}$ that is $<_{L}$-smaller than $y$ and satisfies $P^{z}\downarrow=1$. If the search is successful, then $y\neq\bar{x}$. 
On the other hand, if none is found, then $y$ is the $<_{L}$-smallest element $a$ of $X$
with $a\in L_{\omega_{\omega}^{\text{CK},a}}$, so $y=\bar{x}$.
\end{proof}

\section{Some consequences for the structure of ITRM-degrees}

In the new setting, we can also draw some standard consequences of van Lambalgen's theorem. 

\begin{defini}
If $x\leq_{\text{ITRM}}y$ but not $y\leq_{\text{ITRM}} x$, we write $x<_{\text{ITRM}}y$. If $x\leq_{\text{ITRM}}$ and $y\leq_{\text{ITRM}}x$, then we write $x\equiv_{\text{ITRM}}y$. If neither $x\leq_{\text{ITRM}}y$ nor $y\leq_{\text{ITRM}}x$, 
we call $x$ and $y$ incomparable and write $x\perp_{\text{ITRM}}y$.
\end{defini}

Clearly, $\equiv_{\text{ITRM}}$ is an equivalence relation.
We may hence form, for each real $x$, the $\equiv_{\text{ITRM}}$-equivalence class $[x]_{\text{ITRM}}$ of $x$, called the ITRM-degree of $x$. It is easy to see that $\leq_{\text{ITRM}}$ respects $\equiv_{\text{ITRM}}$,
so that expressions like $[x]_{\text{ITRM}}\leq_{\text{ITRM}}[y]_{\text{ITRM}}$ etc. are well-defined and have the obvious meaning.

\begin{corollary}{\label{ITRMsplitting}}
 If $a$ is ITRM-c-random, $a=a_{0}\oplus a_{1}$, then $a_{0}\perp_{\text{ITRM}} a_{1}$.
\end{corollary}
\begin{proof}
 By Theorem \ref{vLamb}, $a_{0}$ and $a_{1}$ are mutually ITRM-c-random. If $a_{0}$ was ITRM-computable from $a_{1}$, then $\{a_{0}\}$ would be decidable
in the oracle $a_{1}$, meager and contain $a_{0}$, so $a_{0}$ would not be ITRM-c-random relative to $a_{1}$, a contradiction. By symmetry, the claim follows.
\end{proof}

\begin{lemma}{\label{randomnessandhalting}}
Let $h$ be a real coding the halting problem for ITRMs as in the remark following Theorem \ref{weakrandomreducibility}. Then there is an ITRM-c-random real $x\leq_{\text{ITRM}} h$.
\end{lemma}
\begin{proof}
 Given $h$, we can compute a code for $L_{\omega_{\omega}^{\text{CK}}+2}$ as follows: By Theorem 4.6 of \cite{Ca2}, $h$ is ITRM-recognizable. Clearly, $h$ is not ITRM-computable. Hence,
by Theorem 3.2 of \cite{Ca2}, we have $\omega_{\omega}^{\text{CK},h}>\omega_{\omega}^{\text{CK}}$. By Theorem \ref{relITRM}, a real $r$ is ITRM-computable from $h$ if and only if
it is an element of $L_{\omega_{\omega}^{\text{CK},h}}[h]$. As $L_{\omega_{\omega}^{\text{CK}}+3}$ contains a code for $L_{\omega_{\omega}^{\text{CK}}+2}$, such a code is contained
in $L_{\omega_{\omega}^{\text{CK},h}}\subseteq L_{\omega_{\omega}^{\text{CK},h}}[h]$ and is hence ITRM-computable from $h$.

Now $L_{\omega_{\omega}^{\text{CK}+2}}$ will contains a real $x$ which is Cohen-generic over $L_{\omega_{\omega}^{\text{CK}}+1}$.
Such an $x$ can easily be computed from a code for $L_{\omega_{\omega}^{\text{CK}+2}}$ and consequently, $x$ itself is ITRM-computable from $h$. 
By Corollary \ref{randomandcohengeneric}, $x$ is ITRM-c-random.
\end{proof}

We have an analogue to the Kleene-Post-theorem on Turing degrees between $0$ and $0^{\prime}$ (see e.g. Theorem VI.1.2 of \cite{So}) for ITRMs. 

\begin{corollary}{\label{ITRMPost}}
With $h$ as in Lemma \ref{randomnessandhalting}, there are $x_{0},x_{1}$ such that we have $[0]_{\text{ITRM}}<_{\text{ITRM}}[x_{0}]_{\text{ITRM}},[x_{1}]_{\text{ITRM}}\leq h$ and $x_{0}\perp_{\text{ITRM}}x_{1}$. 
In particular, there is a real $x_{0}$ such that $[0]_{\text{ITRM}}<_{\text{ITRM}}[x_{0}]_{\text{ITRM}}<_{\text{ITRM}}h$.
\end{corollary}
\begin{proof}
Pick $x$ as in Lemma \ref{randomnessandhalting}, let $x=x_{0}\oplus x_{1}$, and use Corollary \ref{ITRMsplitting}.
\end{proof}

Using van Lambalgen, we can show much more:

\begin{thm}{\label{degreestructure}}
 Let $\mathfrak{T}$ denote the tree of finite $0$-$1$-sequences, ordered by reverse inclusion. Then $\mathfrak{T}$ embeds in the ITRM-degrees between $0$ and $0^{\prime}$, i.e. there is an injection
$f:\mathfrak{T}\rightarrow\mathfrak{P}(\omega)$ such that, for all $x,y\in\mathfrak{T}$, we have $0<_{\text{ITRM}}f(x),f(y)<_{\text{ITRM}}0^{\prime}_{\text{ITRM}}$ and $f(x)\leq_{\text{ITRM}}f(y)$ if and only if $x\supseteq y$. In particular, there is
an infinite descending sequence of ITRM-degrees between $0$ and $0^{\prime}_{\text{ITRM}}$. 
\end{thm}
\begin{proof}
We shall, for each $s\in^{<\omega}2$, define a real $r_{s}$ in such a way that the set $\{[r_{s}]_{\text{ITRM}}|s\in^{<\omega}2\}$ has the desired property.
 Let $r=r_{()}$ be ITRM-c-random such that $0<_{\text{ITRM}}r<_{\text{ITRM}}0^{\prime}_{\text{ITRM}}$. Assuming that $r_{s}$ is already defined, let $r_{s}=x_{0}\oplus x_{1}$ and set $r_{s0}=x_{0}$, $r_{s1}=x_{1}$.
Clearly, if $s$ is a proper initial segment of $t$, then $r_{s}<_{\text{ITRM}}r_{t}$. We show that furthermore, when $s$ and $t$ are incompatible (i.e. none is an initial segment of the other), then
$r_{s}\perp_{\text{ITRM}}r_{t}$: Let $u=r_{s}\cap r_{t}$ be the longest common initial segment of $r_{s}$ and $r_{t}$. Without loss of generality, assume that $s$ starts with $u0$ and that $t$ starts with $u1$.
Then $r_{u}=r_{u0}\oplus r_{u1}$, so, by Theorem \ref{vLamb}, $u0$ and $u1$ are mutually random. Now suppose that $r_{s}$ and $r_{t}$ are not mutually random, e.g. that $r_{s}$ was not random relative to $r_{t}$. Then, as  $r_{t}$ is recursive
in $r_{u1}$, $r_{s}$ is not random relative to $r_{u1}$. Let $U$ be a meager set containing $r_{s}$ which is ITRM-decidable in the oracle $r_{u1}$.

Let $s^{\prime}$ be a binary string such that $u0s^{\prime}=s$ and let $l$ denote its length. We recursively define a sequence $(U_{i}:0\leq i\leq l)$ of subsets of $[0,1]$ as follows:
Let $U_{0}=0$ and for $\leq i<n$ let $U_{i+1}:=U_{i}\oplus [0,1]$ when the $l-i$th digit of $l$ is $0$ and otherwise $U_{i+1}:=[0,1]\oplus U_{i}$. By an easy induction, we have that
$r_{u0(s^{\prime}|i)}\in U_{l-i}$ for $0\leq i\leq l$, where $s^{\prime}|i$ denotes the restriction of $s^{\prime}$ to its first $i$ many bits. In particular, we have $r_{u0}\in U_{l}$.
Moreover, when $X$ is meager, then so are $X\oplus[0,1]$ and $[0,1]\oplus X$, so as $U_{0}=U$ is meager, it follows by another induction that $U_{i}$ is meager for $0\leq i\leq l$ and in particular that $U_{l}$ is meager. It is also easy to see
inductively that, as $U_{0}=U$ is ITRM-decidable in the oracle $r_{u1}$, so is $U_{i}$ for all $0\leq i\leq l$ and hence in particular for $i=l$. So $U_{l}$ is meager, contains $r_{u0}$ and is ITRM-decidable in the oracle $r_{u1}$,
contradicting the fact that $r_{u0}$ is ITRM-c-random relative to $r_{u1}$.
\end{proof}

\begin{corollary}
The ITRM-degrees of ITRM-c-random reals contain no minimal or maximal element under ITRM-reduction. 
\end{corollary}
\begin{proof}
Let $y$ is ITRM-c-random, $y=y_{1}\oplus y_{2}$.

Then first, by Theorem \ref{vLamb}, $y_{1}$ is ITRM-c-random and $y_{1}<_{\text{ITRM}}y$. Hence there is no $<_{\text{ITRM}}$-minimal ITRM-c-random real $y$.

Furthermore, we have $\omega_{\omega}^{\text{CK},y}=\omega_{\omega}^{\text{CK}}$ by Theorem \ref{Admpreserv}, so $L_{\omega_{\omega}^{\text{CK},y}}[y]=L_{\omega_{\omega}^{\text{CK}}}[y]$;
Let $x$ be Cohen-generic over $L_{\omega_{\omega}^{\text{CK}}}[y]$, then, by the relativized version of Lemma \ref{randomandcohengeneric} (see the remark following the Lemma), $x$
is ITRM-c-random relative to $y$. By Theorem \ref{vLamb}, $z:=x\oplus y$ is ITRM-c-random, and we clearly have $y<_{\text{ITRM}}z$, so $y$ is not maximal.

\end{proof}

In the Turing degrees, it is a striking experience that intermediate degrees (i.e. degrees lying properly between two iterations of the Turing jump) don't seem to come up `naturally', but need to
be constructed on purpose (see e.g. the discussion in \cite{Si}). Concerning ITRMs, we can to a certain extent prove that something similar is going on: Namely, reals of intermediate degree are never
recognizable.


\begin{thm}{\label{recogchar}}
If $x\subseteq\omega$ is recognizable, then $x\in L_{\omega_{\omega}^{\text{CK},x}}$.
\end{thm}
\begin{proof}
 Let $P$ be a program recognizing $x$. Then, by Theorem \ref{hp}, $P^{x}$ stops in less than $\omega_{\omega}^{\text{CK},x}$ many steps. Consequently, the computation is contained in $V_{\omega_{\omega}^{\text{CK},x}}$.
Hence $V_{\omega_{\omega}^{\text{CK},x}}\models\exists{y}P^{y}\downarrow=1$. Clearly, $\omega_{\omega}^{\text{CK},x}$ is a limit of admissible ordinals. By a theorem of Jensen and Karp (\cite{JeKa}), $\Sigma_1$-formulas
are absolute between $L_{\alpha}$ and $V_{\alpha}$ whenever $\alpha$ is a limit of admissibles. Hence $L_{\omega_{\omega}^{\text{CK},x}}\models\exists{y}P^{y}\downarrow=1$. By absoluteness of computations then,
we have $x\in L_{\omega_{\omega}^{\text{CK},x}}$, as desired.
\end{proof}

\begin{lemma}{\label{haltrecog}}
 If $x$ is recognizable, then so is $x\oplus x^{\prime}_{\text{ITRM}}$.
\end{lemma}
\begin{proof}
This is a relativized version of Theorem $26$ of \cite{Ca2}. We sketch the proof for completeness: By assumption, let $P$ be an ITRM-program that recognizes $x$.

It then follows that $\omega_{\omega}^{\text{CK},x^{\prime}}>\omega_{\omega}^{\text{CK},x}$: The idea is to first use $x^{\prime}_{\text{ITRM}}$ to decide for any program $P_{i}$ whether 
$P_{i}^{x}$ computes a real coding a well-ordering; moreover, these codes can be read out explictly with the help of $x^{\prime}$. To see this, let $P_{i}$ be given. It is easy to compute from $i$ the index $f(i)$
for an ITRM-program $P_{f(i)}$ such that $P_{f(i)}^{x}$ halts if and only if $P_{i}^{x}$ halts with output $0$ or $1$, i.e. computes a real number. Given that $P_{i}^{x}$ computes a real number, as ITRMs can check
relations for well-foundedness (see \cite{ITRM}), it is also easy to effectively obtain from $i$ an ITRM-program index $g(i)$ such that $P_{g(i)}^{x}$ halts if and only if the real computed by $P_{i}^{x}$ codes a well-founded relation.
Furthermore, given that $P_{i}^{x}$ computes a real number, it is also easy to effectively obtain from $i$ an ITRM-program index $h(i)$ such that for every $j\in\omega$, $P_{h(i)}^{x}(j)$ halts if and only if the $j$th bit of the
real number computed by $P_{i}^{x}$ is $1$. Hence, using $x^{\prime}$, we can compute the $j$th bit of the $i$th code of a well-ordering ITRM-computable in the oracle $x$.

Finally, we can now assemble all of these codes into a code for a well-ordering of order type their sum, which will be the supremum of the ordinals with codes ITRM-computable in the oracle $x$, i.e. $\omega_{\omega}^{\text{CK},x}$. Thus, from
$x^{\prime}$, we can compute a code for a well-ordering of order-type $\omega_{\omega}^{\text{CK},x}$.

As $\omega_{\omega}^{\text{CK},x^{\prime}}$ is the supremum of the ordinals with codes ITRM-computable in the oracle $x^{\prime}$, it follows that $\omega_{\omega}^{\text{CK},x^{\prime}}>\omega_{\omega}^{\text{CK},x}$;
hence there is some minimal $k\in\omega$ such that $\omega_{k}^{\text{CK},x^{\prime}}>\omega_{\omega}^{\text{CK},x}$. There is a real $c$ coding $L_{\omega_{\omega}^{\text{CK},x}}[x]$ contained in $L_{\omega_{\omega}^{\text{CK},x}+1}[x]$ and hence
in $L_{\omega_{k}^{\text{CK},x^{\prime}}}[x^{\prime}]$ (as $x$ is computable from $x^{\prime}$). By Theorem \ref{relITRM} above, there is an ITRM-program $Q$ computing $c$ in the oracle $x^{\prime}$. But as $\omega_{\omega}^{\text{CK},x}$ is the supremum
of the halting times for ITRM-programs in the oracle $x$, a program in the oracle $x$ will halt inside $L_{\omega_{\omega}^{\text{CK},x}}[x]$ if and only if it halts in the real world. Checking whether some ITRM-program $P_{i}^{x}$ halts
hence reduces to checking a first-order statement in the structure coded by $c$, which can be done as described in the last section of \cite{ITRM2}. $x^{\prime}$ can now be identified by checking whether the statement `$P_{i}^{x}\downarrow$' holds
in the structure coded by $c$ for each $i\in\omega$ and comparing the results to the oracle.

This leads to the following procedure for recognizing $x\oplus x^{\prime}$: Suppose that $y=y_{1}\oplus y_{2}$ is given in the oracle. First, we run $P$ on $y_{1}$ to check whether $y_{1}=x$. If not, then $y\neq x\oplus x^{\prime}$.
Otherwise, check, using a halting problem solver for $Q$, whether $Q^{y_{2}}$ computes a real $r$ coding $L_{\omega_{\omega}^{\text{CK},y_{1}}}[y_{1}]$. If not, then $y\neq x\oplus x^{\prime}$. Otherwise, using $r$, check for each $i\in\omega$
whether $L_{\omega_{\omega}^{\text{CK},x}}[x]\models P_{i}^{x}\downarrow\leftrightarrow i\in y_{2}$. If not, then $y\neq x\oplus x^{\prime}$, otherwise we have $y=x\oplus x^{\prime}$.
\end{proof}

\begin{corollary}{\label{hpitrecog}}
 If $x\subseteq\omega$ is recognizable, then so is $x^{\prime}_{\text{ITRM}}$.
\end{corollary}
\begin{proof}
Note that $x$ is Turing-computable from $x^{\prime}_{\text{ITRM}}$: To determine the $i$th bit of $x$ when $x^{\prime}_{\text{ITRM}}$ is given, simply use $x^{\prime}_{\text{ITRM}}$ to 
check whether the ITRM-program $Q_{i}$ that stops when the $i$th bit of its oracle is $1$ and otherwise enters 
an infinite loop halts in the oracle $x$. This works uniformly in $x$. Let $Q$ be an ITRM-program that computes $x$ from $x^{\prime}_{\text{ITRM}}$, for every $x\subseteq\omega$.

To recognize $x^{\prime}_{\text{ITRM}}$, we proceed as follows: First, let $P$ be a program that recognizes $x$. Now, given $y$ in the oracle, first check, using a halting problem solver for $Q$, whether $Q^{y}$ computes a real number.
If not, then $y\neq x^{\prime}$. Otherwise, let $z$ be that number and run $P^{z}$. If $P^{z}\downarrow=0$, then $z\neq x$, so $y\neq x^{\prime}_{\text{ITRM}}$ (since $Q$ would have computed $x$ from $x^{\prime}_{\text{ITRM}}$).
Otherwise, we have $z=x$ and can use Lemma \ref{haltrecog} to check whether $z\oplus y=x\oplus x^{\prime}_{\text{ITRM}}$, which will be the case if and only if $y=x^{\prime}_{\text{ITRM}}$.
\end{proof}

\begin{thm}{\label{iteratedhp}}
 For $i\in\omega$, we have $\omega_{\omega}^{\text{CK},0^{(i)}_{\text{ITRM}}}=\omega_{\omega(i+1)}^{\text{CK}}$; consequently, $i+1$ is the minimal $n\in\omega$ such that $0_{\text{ITRM}}^{(i)}\in L_{\omega_{\omega n}^{\text{CK}}}$.
\end{thm}
\begin{proof}
For $i\in\omega$, let $\sigma_{i}$ denote $\omega_{\omega}^{\text{CK},0^{(i)}_{\text{ITRM}}}$.

It is easy to see that $0^{(i)}_{\text{ITRM}}$ is ITRM-computable from $0^{(i+1)}_{\text{ITRM}}$ for all $i\in\omega$. Hence, by Lemma \ref{haltrecog}, $0^{(i)}_{\text{ITRM}}$ is recognizable for each $i\in\omega$. By Theorem \ref{recogchar} then, 
$0^{(i)}_{\text{ITRM}}\in L_{\sigma_{i}}$ for each $i\in\omega$; thus
$L_{\sigma_{i}}[0^{(i)}_{\text{ITRM}}]=L_{\sigma_{i}}$. Hence a real $x$ is computable from $0^{(i)}_{\text{ITRM}}$ if and only if $x\in L_{\sigma_{i}}$.
As $0^{(i+1)}_{\text{ITRM}}$ is not ITRM-computable from $0^{(i)}_{\text{ITRM}}$, we must have $\omega_{\omega}^{\text{CK},0^{(i+1)}_{\text{ITRM}}}>\omega_{\omega}^{\text{CK},0^{(i)}_{\text{ITRM}}}$ for all $i\in\omega$;
hence $(\sigma_{i}|i\in\omega)$ is a strictly increasing sequence of limits of admissible ordinals.

We now proceed inductively to show that $\sigma_{i}=\omega_{\omega(i+1)}^{\text{CK}}$:

As $0^{\prime}_{\text{ITRM}}$ is not ITRM-computable, it is not an element of $L_{\omega_{\omega}^{\text{CK}}}$. As it is definable over $L_{\omega_{\omega}^{\text{CK}}}$, it is
an element of $L_{\omega_{\omega}^{\text{CK}}+1}$ and hence of $L_{\omega_{\omega2}^{\text{CK}}}$. Hence $\omega_{\omega}^{\text{CK}}<\sigma_{1}\leq\omega_{\omega2}^{\text{CK}}$; as $\sigma_{1}$ is a limit of admissibles,
we have $\sigma_{1}=\omega_{\omega2}^{\text{CK}}$.

Similarly, assuming that $\sigma_{i}=\omega_{\omega(i+1)}^{\text{CK}}$, $0^{(i+1)}_{\text{ITRM}}$ is definable over $L_{\omega_{\omega(i+1)}^{\text{CK}}}$ and hence an element of $L_{\omega_{\omega(i+2)}^{\text{CK}}}$.
Hence $\omega_{\omega(i+1)}^{\text{CK}}<\sigma_{i+1}\leq\omega_{\omega(i+2)}^{\text{CK}}$, so $\sigma_{i+1}=\omega_{\omega(i+2)}^{\text{CK}}$, as desired.
\end{proof}

We can now extend Theorem \ref{degreestructure} to obtain intermediate degrees also for iterations of the jump operator:

\begin{corollary}
 There is a real $x$ such that $0^{\prime}_{\text{ITRM}}<_{\text{ITRM}}x<_{\text{ITRM}}0^{\prime\prime}_{\text{ITRM}}$.
\end{corollary}
\begin{proof}
 Let $x\in L_{\omega_{\omega3}^{\text{CK}}}-L_{\omega_{\omega2}^{\text{CK}}}$ be Cohen-generic over $L_{\omega_{\omega2}^{\text{CK}}}$ (and hence ITRM-c-random). (Such a real exists because, at $\omega_{\omega2}^{\text{CK}}+1$, the ultimate projectum drops to $\omega$,
so that already $L_{\omega_{\omega2}^{\text{CK}}+2}$ contains a real Cohen-generic over $L_{\omega_{\omega2}^{\text{CK}}}$; this real is certainly Cohen-generic over $L_{\omega_{\omega}^{\text{CK}}}$ and hence ITRM-c-random
by Theorem \ref{randomandcohengeneric}, and cannot be an element of $L_{\omega_{\omega2}^{\text{CK}}}$.)

So $x\in L_{\omega_{\omega3}^{\text{CK}}}=L_{\omega_{\omega}^{\text{CK},0^{\prime\prime}_{\text{ITRM}}}}[0^{\prime\prime}_{\text{ITRM}}]$ by Lemma \ref{haltrecog} and Corollary \ref{hpitrecog}, hence $x\leq_{\text{ITRM}}0^{\prime\prime}_{\text{ITRM}}$. 
But we cannot have
$x\in L_{\omega_{\omega}^{\text{CK},0^{\prime}_{\text{ITRM}}}}[0^{\prime}_{\text{ITRM}}]=L_{\omega_{\omega}^{\text{CK},0^{\prime}_{\text{ITRM}}}}$ since $L_{\omega_{\omega}^{\text{CK},0^{\prime}_{\text{ITRM}}}}=L_{\omega_{\omega2}^{\text{CK}}}$ by
Theorem \ref{iteratedhp}. Consequently, we have $x\leq 0_{\text{ITRM}}^{\prime\prime}$ and $x\not\leq0^{\prime}_{\text{ITRM}}$. 

As $x$ is Cohen-generic over $L_{\omega_{\omega2}^{\text{CK}}}$ and $L_{\omega_{\omega2}^{\text{CK}}}=L_{\omega_{\omega}^{\text{CK},0^{\prime}_{\text{ITRM}}}}[0^{\prime}_{\text{ITRM}}]$ by Theorem \ref{iteratedhp},
we have 
$\omega_{\omega}^{\text{CK},x\oplus 0^{\prime}_{\text{ITRM}}}=\omega_{\omega}^{\text{CK},0^{\prime}_{\text{ITRM}}}=\omega_{\omega2}^{\text{CK}}<\omega_{\omega3}^{\text{CK}}=\omega_{\omega}^{\text{CK},0^{\prime\prime}_{\text{ITRM}}}$. In particular,
there is a real coding a well-ordering of order type $\omega_{\omega3}^{\text{CK}}$ ITRM-computable from $0^{\prime\prime}_{\text{ITRM}}$, but not from $x\oplus 0^{\prime}_{\text{ITRM}}$. Hence
$0^{\prime\prime}_{\text{ITRM}}$ cannot be ITRM-computable from $x\oplus 0^{\prime}_{\text{ITRM}}$, so that $x\oplus0^{\prime}_{\text{ITRM}}\leq_{\text{ITRM}}0^{\prime\prime}_{\text{ITRM}}$ implies
 $x\oplus0^{\prime}_{\text{ITRM}}<_{\text{ITRM}}0^{\prime\prime}_{\text{ITRM}}$.

Hence $0^{\prime}_{\text{ITRM}}<_{\text{ITRM}}x\oplus 0^{\prime}_{\text{ITRM}}<_{\text{ITRM}}0_{\text{ITRM}}^{\prime\prime}$, so
$y:=x\oplus 0^{\prime}_{\text{ITRM}}$ is as desired.
\end{proof}

By a similar argument, one gets intermediate degrees between any two successive (finite) iterations of the ITRM-jump:

\begin{corollary}{\label{intermediatedegrees}}
 For every $i\in\omega$, there is a real $x_{i}$ such that $0^{(i)}_{\text{ITRM}}<_{\text{ITRM}}x<0^{(i+1)}_{\text{ITRM}}$.
\end{corollary}

Given Corollary \ref{intermediatedegrees}, one may ask whether any of these intermediate degrees is in some natural sense `unique'. As above (see the passage preceeding Definition \ref{recogdef}), a way to capture the meaning of `unique' that suggests
itself in the context of ITRMs is ITRM-recognizability. It turns out that, in this interpretation, no intermediate degree in the finite iterations of the ITRM-jump is `unique':

\begin{thm}{\label{finiteintermediaterecog}}
 Assume that $x\leq_{\text{ITRM}}0_{\text{ITRM}}^{(i)}$ for some $i\in\omega$ and that $x\in$ RECOG. Then there is $j\in\omega$ such that $x\in [0_{\text{ITRM}}^{(j)}]_{\text{ITRM}}$.
\end{thm}
\begin{proof}
We start by showing: If $x>_{\text{ITRM}}0_{\text{ITRM}}^{(i)}$ is recognizable, then $x\geq_{\text{ITRM}}0_{\text{ITRM}}^{(i+1)}$. To see this, note that $x>0_{\text{ITRM}}^{(i)}$ implies
that $x\notin L_{\omega_{\omega(i+1)}^{\text{CK}}}$ by Theorem \ref{iteratedhp}. 
 As $x$ is recognizable, we have $x\in L_{\omega_{\omega}^{\text{CK},x}}$ by
Theorem \ref{recogchar}. Hence $\omega_{\omega}^{\text{CK},x}$ is a limit of admissibles strictly bigger than $\omega_{\omega(i+1)}^{\text{CK}}$,
so $\omega_{\omega}^{\text{CK},x}\geq\omega_{\omega(i+2)}^{\text{CK}}$. As $0_{\text{ITRM}}^{(i+1)}\in L_{\omega_{\omega(i+2)}^{\text{CK}}}$,
we have $0_{\text{ITRM}}^{(i+1)}\in L_{\omega_{\omega}^{\text{CK},x}}[x]$, and hence $0_{\text{ITRM}}^{(i+1)}\leq x$.
Now assume for a contradiction that $x$ is as in the assumption of the theorem, but not ITRM-computably equivalent to some $0^{(i)}_{\text{ITRM}}$. In particular then, $x$ is not ITRM-computable (otherwise $x\in [0_{\text{ITRM}}^{(0)}]_{\text{ITRM}}$).
So $x>_{\text{ITRM}}0_{\text{ITRM}}^{(0)}$. As $x\leq_{\text{ITRM}}0^{(i)}_{\text{ITRM}}$ for some $i\in\omega$, there is some $j\in\omega$ such that $x\not>_{\text{ITRM}}0^{(j)}$; without loss of generality, let $j$ be minimal with this property.
Then $j>0$ and $x>_{\text{ITRM}}0^{(j-1)}_{\text{ITRM}}$, so $x\geq_{\text{ITRM}}0^{(j)}_{\text{ITRM}}$. As $x>_{\text{ITRM}}0^{(j)}_{\text{ITRM}}$ is excluded, we have $x=_{\text{ITRM}}0^{(j)}_{\text{ITRM}}$, so $j$ is as desired.
\end{proof}

We can, however, also find reals $x$ with degree strictly between $0$ and $0_{\text{ITRM}}^{\prime}$ that are not random. This will, after some preparation, be shown in Theorem \ref{nonrandomintermediate} below.

\begin{defini}
 For $x,y\subseteq\omega$, we write $x\leq_{\text{ITRM}}^{n}y$ to indicate that $x$ is computable from $y$ by an ITRM-program using at most $n$ registers.
\end{defini}

\begin{lemma}{\label{boundingreduction}}
There is a natural number $C$ such that, for all reals $x,y$ with $x\leq_{\text{ITRM}}^{n}y$, we have $\omega_{k}^{CK,x}\leq\omega_{Ck+n+1}^{\text{CK},y}$ for all $k\in\omega$.
\end{lemma}
\begin{proof}
 It is not hard to see that there is a constant $c$ such that, for every $0<k\in\omega$ and every real $z$, a real coding $\omega_{k}^{\text{CK},z}$ is (uniformly) ITRM-computable from $z$
by a program $P$ using at most $kc$ many registers. Now, if $x\leq_{\text{ITRM}}^{n}y$ via a program $Q$ using at most $n$ registers, we can, given $y$ in the oracle, first use $Q$
to compute $x$ and then $P$ to compute a code for $\omega_{k}^{\text{CK},x}$. Hence, for some constant $d$ (which is independent from $k,x,y,n$), we can compute a code for $\omega_{k}^{\text{CK},x}$
in the oracle $y$ using at most $kc+n+d$ many registers. With $C=c+d$, we hence get that such a code is computable from $y$ with at most $kC+n$ many registers.
However, every such real will be an element of $L_{\omega_{kC+n+1}^{\text{CK},y}}[y]$ and hence cannot code an ordinal greater than $\omega_{kC+n+1}^{\text{CK},y}$.
Thus $\omega_{k}^{\text{CK},x}\leq\omega_{Ck+n+1}^{\text{CK},y}$.
\end{proof}

We recall the following special case of Theorem $4.1$ from \cite{SW}: 

\begin{defini}{\label{index}}
An ordinal $\alpha<\omega_{1}^{L}$ is an index if and only if $L_{\alpha+1}-L_{\alpha}$ contains a real number. 
\end{defini}

\begin{thm}{\label{simpweit}}
Let $(\alpha_{\iota}|\iota<\delta)$ be a countable sequence of admissible ordinals greater than $\omega$ such that, for each $\nu<\delta$, $\alpha_{\nu}$
is admissible relative to $(\alpha_{\iota}|\iota<\nu)$. Then there is a real $x$ such that $\alpha_{\iota}$ is the $\iota$th $x$-admissible ordinal greater than $\omega$.
With $\alpha=\text{sup}\{\alpha_{\iota}|\iota<\delta\}$, such a real is arithmetically constructible from any real coding $(L_{\alpha+1}(\{\alpha_{\iota}|\iota<\delta\}),\in,\{\alpha_{\iota}|\iota<\delta\})$.
If $\alpha$ is an index and $(\alpha_{\iota}|\iota<\delta)$ is definable over $L_{\alpha}$, then such a real is contained in $L_{\alpha+2}$.
\end{thm}

\begin{thm}{\label{nonrandomintermediate}}
 There is a real $0<_{\text{ITRM}}x<_{\text{ITRM}}0^{\prime}_{\text{ITRM}}$ such that $x$ is not ITRM-c-random; in fact, $x$ can be chosen in such a way that $x$ is not ITRM-reducible to any ITRM-c-random real $y$.
\end{thm}
\begin{proof}
 By Theorem \ref{simpweit}, there is a real $x$ such that $\omega_{i}^{\text{CK},x}=\omega_{i^{2}}^{\text{CK}}$ for all $i\in\omega$, definable over $L_{\omega_{\omega}^{\text{CK}}+1}$ and hence an element
of $L_{\omega_{\omega}^{\text{CK}}+2}$. From $0^{\prime}_{\text{ITRM}}$, we can compute a real coding $L_{\omega_{\omega}^{\text{CK}}+2}$ and therefore also $x$, hence $x\leq_{\text{ITRM}}0^{\prime}_{\text{ITRM}}$.
As $\omega_{\omega}^{\text{CK},x}=\omega_{\omega}^{\text{CK}}<\omega_{\omega}^{\text{CK},0^{\prime}}$ by Theorem \ref{iteratedhp}, $0^{\prime}_{\text{ITRM}}$ is not computable from $x$: for otherwise, a code
for $\omega_{\omega}^{\text{CK}}$ would be computable from $x$.
Now assume that $x$ is reducible to an ITRM-c-random real $y$ via a program $P$ using $n$ registers. By Corollary \ref{Admpreserv}, we have $\omega_{i}^{\text{CK},y}=\omega_{i}^{\text{CK}}$ for 
all $i\in\omega$. By Lemma \ref{boundingreduction}, there is $1<C\in\omega$ such that $\omega_{i}^{\text{CK},x}\leq\omega_{Ci+n+1}^{\text{CK},y}$ for all $i\in\omega$.
Consequently, we have $\omega_{i^{2}}^{\text{CK}}=\omega_{i}^{\text{CK},x}\leq\omega_{Ci+n+1}^{\text{CK},y}=\omega_{Ci+n+1}^{\text{CK}}$ for all $i\in\omega$, which implies $i^{2}\leq Ci+n+1$ for all $i\in\omega$, which
is obviously false (e.g. for $i\geq C+n+1$). So $x$ is not reducible to an ITRM-c-random real.
\end{proof}

\begin{corollary}
 There are $0<_{\text{ITRM}}x,y<0^{\prime}_{\text{ITRM}}$ such that $x\perp_{\text{ITRM}}y$ and neither $x$ nor $y$ is ITRM-c-random.
\end{corollary}
\begin{proof}
Let $(a_{n}|n\in\omega)$ and $(b_{n}|n\in\omega)$ be two strictly increasing sequences of natural numbers such that, for every $C\in\omega$, there are $k,l\in\omega$ with
$a_{k}>b_{Ck}$ and $b_{l}>a_{Cl}$. For example, we may take $a_{1}=b_{1}=1$, and then recursively $a_{(2n+1)!}=b_{(2n+1)!}+(n-1)(2n+1)!+1$,
$b_{(2n)!}=a_{(2n)!}+(n-1)(2n)!+1$ and $a_{k+1}=a_{k}+1$, $b_{k+1}=b_{k}+1$ if $k+1$ is not of the form $(2n+1)!$ or $(2n)!$, respectively.
Now, by Theorem \ref{simpweit}, we find reals $x,y$ so that $\omega_{i}^{\text{CK},x}=a_{i}$ and $\omega_{i}^{\text{CK},y}=b_{i}$ for all $i\in\omega$. $x,y$ are ITRM-computable
from $0^{\prime}_{\text{ITRM}}$ and not ITRM-c-random by the same argument as in Theorem \ref{nonrandomintermediate}. To see that $x\perp_{\text{ITRM}}y$, assume for a contradiction that
one is reducible to the other, without loss of generality $x\leq^{n}_{\text{ITRM}}y$. By Lemma \ref{boundingreduction}, there is hence $C\in\omega$ such that $\omega_{i}^{\text{CK},x}\leq\omega_{Ci+n+1}^{\text{CK},y}$ for all $i\in\omega$.
Pick $0<k\in\omega$ such that $a_{k}>b_{(C+n+1)k}$. Then $\omega_{(C+n+1)k}^{\text{CK},y}=\omega_{b_{(C+n+1)k}}^{\text{CK}}<\omega_{a_{k}}^{\text{CK}}=\omega_{k}^{\text{CK},x}\leq\omega_{Ck+n+1}^{\text{CK},y}$; consequently,
we have $(C+n+1)k<Ck+n+1$, which is impossible for $k>0$.
\end{proof}

\noindent
\textbf{Remark}: The same strategy allows the construction of arbitrarily long finite and in fact a countable set of pairwise incomparable, reals $<_{\text{ITRM}}0^{\prime}_{\text{ITRM}}$ which are
neither ITRM-c-random nor even ITRM-reducible to a ITRM-c-random real.


\section{Autoreducibility for Infinite Time Register Machines}

Intuitively, a real $x$ is called autoreducible if each of its bits can be effectively recovered from its position in $x$ and all other bits of $x$. A discussion of the classical notion of autoreducibility can, for example, be found in \cite{DoHi}.
We want to consider how this concept behaves in the context of infinitary machine models of computation.
For the time being, we focus on ITRMs - but the notion of course
makes sense for other types like the Infinite Time Turing Machines (ITTMs, see \cite{HaLe}), Ordinal Turing machines and Ordinal Register Machines (OTMs and ORMs, see e.g. \cite{Ko2}) etc. as well.

\begin{defini}
For $x\in ^{\omega}2$, we define $x_{\setminus n}$ as $x$ with its $n$th bit deleted (i.e. the bits up to $n$ are the same, all further bits are shifted one place to the left).
We say that $x$ is ITRM-autoreducible if and only if there is an ITRM-program $P$ such that $P^{x_{\setminus n}}(n)\downarrow=x(n)$ for all $n\in\omega$.
$x$ is called totally incompressible 
if and only if it is not ITRM-autoreducible, i.e. there is no ITRM-program $P$ such that $P^{x_{\setminus n}}(n)\downarrow=x(n)$ for all $n\in\omega$. If there is such a program,
then we say that $P$ autoreduces $x$, $P$ is an autoreduction for $x$ or that $x$ is autoreducible via $P$.
\end{defini}

\begin{corollary}
 No totally incompressible $x$ is ITRM-computable or even recognizable. $0^{\prime}_{\text{ITRM}}$, the real coding the halting problem for ITRMs, is ITRM-autoreducible.
\end{corollary}
\begin{proof}
 Clearly, if $P$ computes $x$, then $P$ is also an autoreduction  for $x$. If $x$ is recognizable and $P$ recognizes $x$, we can easily retrieve a 
deleted bit by plugging in $0$ and $1$ and letting $P$ run on both results to see for which one $P$ stops with output $1$. (The same idea works for
finite subsets instead of single bits.) For $0^{\prime}_{\text{ITRM}}$, if a program index $j$ is given, it is easy to determine
some index $i\neq j$ corresponding to a program that works in exactly the same way (by e.g. adding a meaningless line somewhere), so that the remaining
bits allow us to reconstruct the $j$th bit.
\end{proof}

\noindent
\textbf{Remark}: The autoreducibility of $0^{\prime}_{\text{ITRM}}$ also follows from the first part of the corollary and the recognizability of $0^{\prime}_{\text{ITRM}}$ (see \cite{Ca2}).

\begin{defini}
 Let $x\in^{\omega}2$, $i\in\omega$. Then $\text{flip}(x,i)$ denotes the real obtained from $x$ by just changing the $i$th bit, i.e. $x\Delta\{i\}$.
\end{defini}

In the classical setting, no random real is autoreducible. This is still true for ITRM- and ITRM-c-randomness:

\begin{thm}{\label{randomnessimpliestotalincompressibility}}
 If $x$ is ITRM-random or ITRM-c-random, then $x$ is totally incompressible. 
\end{thm}
\begin{proof}
 Assume that $x$ is autoreducible via $P$. We show that $x$ is not ITRM-random. Let $X$ be the set of all $y$ which are autoreducible via $P$.
Obviously, we have $x\in X$. 
$X$ is certainly decidable: Given $y$, use a halting problem solver for $P$ to see whether $P^{y_{\setminus n}}(n)\downarrow$ for all $n\in\omega$.
If not, then $y\notin X$. Otherwise, carry out these $\omega$ many computations and check the results one after the other.

Since $X$ is ITRM-decidable, it is provably $\Delta_{2}^{1}$, which implies that $X$ 
is measurable. 

We show that $X$ must be of measure $0$. To see this, assume for a contradiction that $\mu(X)>0$.
Note first, that, whenever $y$ is $P$-autoreducible and $z$ is a real that deviates from $y$ in exactly one digit (say, the $i$th bit), then $z$ is not $P$-autoreducible
(since $P$ will compute the $i$th bit wrongly).

By the Lebesgue density theorem, there is an open basic interval $I$ (i.e. consisting of all reals that start with a certain finite binary string $s$ length $k\in\omega$) such that
the relative measure of $X$ in $I$ is $>\frac{1}{2}$. Let $X^{\prime}=X\cap I$, and let $X^{\prime}_0$ and $X^{\prime}_1$ be the subsets of $X^{\prime}$ consisting 
of those elements that have their $(k+1)$th digit equal to $0$ or $1$, respectively. Clearly, $X^{\prime}_{0}$ and $X^{\prime}_{1}$ are measurable, $X^{\prime}_{0}\cap X^{\prime}_{1}=\emptyset$
and $X^{\prime}=X^{\prime}_{0}\cup X^{\prime}_{1}$. Now define $\bar{X}^{\prime}_{0}$ and $\bar{X}^{\prime}_{1}$ by changing the $(k+1)$th bit of all elements of $X^{\prime}_{0}$ and $X^{\prime}_{1}$,
respectively. Then all elements of $\bar{X}^{\prime}_{0}$ and $\bar{X}^{\prime}_{1}$ are elements of $I$ (as we have not changed the first $k$ bits), none of them is $P$-autoreducible (since they all
deviate from $P$-autoreducible elements by exactly one bit, namely the $k$th), $\bar{X}^{\prime}_{0}\cap\bar{X}^{\prime}_{1}=\emptyset$ (elements of the former set have $1$ as their $(k+1)$th digit,
for elements of $\bar{X}^{\prime}_{1}$ it is $0$) and $\mu(\bar{X}^{\prime}_{0})=\mu(X^{\prime}_{0})$, $\mu(\bar{X}^{\prime}_{1})=\mu(X^{\prime}_{1})$ (as the $\bar{X}^{\prime}_{i}$ are just 
translations of the $X^{\prime}_{i}$). As no element of the $\bar{X}^{\prime}_{i}$ is $P$-autoreducible, we have $(\bar{X}^{\prime}_{0}\cup\bar{X}^{\prime}_{1})\cap X^{\prime}=\emptyset$. Let
$\bar{X}^{\prime}:=\bar{X}^{\prime}_{0}\cup\bar{X}^{\prime}_{1}$.

Then we have
 $\mu_{I}(\bar{X}^{\prime})=\mu_{I}(\bar{X}^{\prime}_{0}\cup\bar{X}^{\prime}_{1})=\mu_{I}(\bar{X}^{\prime}_{0})+\mu_{I}(\bar{X}^{\prime}_{1})=\mu_{I}(X^{\prime}_{0})+\mu_{I}(X^{\prime}_{1})=\mu_{I}(X^{\prime})>\frac{1}{2}$ (where
$\mu_{I}$ denotes the relative measure for $I$). So $X^{\prime}$ and $\bar{X}^{\prime}$ are two disjoint subsets of $I$ both with relative measure $>\frac{1}{2}$, a contradiction.

For the analogous statement for ITRM-c-randomness, we proceed similarly, taking $I$ to be an interval in which $X\cap I$ is comeager instead. That such an $I$ exists can be seen as follows:
Suppose that $X$ is not meager. As above, $X$ is ITRM-decidable, hence provably $\Delta_{2}^{1}$ and therefore has the Baire property. Then,
there is an open set $U$
such that $(X\setminus U)\cup (U\setminus X)$ is meager. In particular, $U$ is not empty. Hence $X$ is comeager in $U$. As $U$ is open, there is a nonempty open interval $I\subseteq U$.
It is now obvious that $X\cap I$ is comeager in $I$, so $I$ is as desired. We then use the same argument as above, noting that two comeager subsets of $I$ cannot be disjoint.
\end{proof}

\begin{corollary}{\label{Cohengenericsareincompressible}}
 Let $x$ be Cohen-generic over $L_{\omega_{\omega}^{\text{CK}}}$. Then $x$ is totally incompressible.
\end{corollary}
\begin{proof}
This follows immediately from Theorem \ref{randomnessimpliestotalincompressibility} and Corollary \ref{randomandcohengeneric}.
\end{proof}

\noindent
\textbf{Remark}: When we demand that every finite subset of the bits of $x$ (instead of every single bit) can be effectively obtained from the remaining bits and the positions of the missing bits by some ITRM-program, 
we get the notion of \textbf{strong} autoreducibility. Strong autoreducibility is a strictly stronger notion than autoreducibility: If $x$
is Cohen-generic over $L_{\omega_{\omega}^{\text{CK}}}$, then $y:=x\oplus x$ is clearly autoreducible; however, by Corollary \ref{Cohengenericsareincompressible}, $y$ is not strongly autoreducible, as a procedure for
obtaining even just the $2n$th and $2n+1$th bit of $y$ from the remaining bits for every $n$ would lead to an autoreduction for $x$. 

\begin{defini}
 Denote by IC$_{\text{ITRM}}$ and RA$_{\text{ITRM}}$ the set of totally incompressible and ITRM-random reals, respectively.
\end{defini}

\begin{corollary}{\label{totallycompressiblesarerare}}
The set of ITRM-autoreducible reals has measure $0$ and is meager.
\end{corollary}
\begin{proof}
 The proof of Theorem \ref{randomnessimpliestotalincompressibility} shows that, for any ITRM-program $P$, the set of reals autoreducible via $P$ has measure $0$. As there are only countable
many programs, the result follows. Meagerness follows from Corollary \ref{Cohengenericsareincompressible}.
\end{proof}

By Theorem \ref{randomnessimpliestotalincompressibility}, we have RA$_{\text{ITRM}}\subseteq $IC$_{\text{ITRM}}$. By now, we have observed various similarities between totally incompressibility
and randomness. However, the converse of Theorem \ref{randomnessimpliestotalincompressibility} fails for ITRM-randomness:

\begin{prop}{\label{totalincompressibilitydoesnotimplyrandomness}}
 IC$_{\text{ITRM}}\not\subseteq$ RA$_{\text{ITRM}}$, i.e. there is a real $x$ such that $x$ is totally incompressible, but not ITRM-random.
\end{prop}
\begin{proof}
Let $y$ be ITRM-c-random and let $x=y\oplus 0$. Then $x$ is not ITRM-c-random since $[0,1]\oplus 0$ is ITRM-decidable and meager. 
Moreover, $x$ is totally incompressible: For an autoreduction $P$ for $x$ would immediately lead to an autoreduction for $y$: To determine the $n$th bit of $y$ given $y_{\setminus{n}}$,
we only need to run $P^{(y\oplus 0)_{\setminus{2n}}}$, where $(y\oplus 0)_{\setminus{2n}}$ is just $y_{\setminus{n}}\oplus 0$ with an extra $0$ inserted at the $2n$th place.
However, by Theorem \ref{randomnessimpliestotalincompressibility}, $y$ is totally incompressible.
\end{proof}

\section{Conclusion and further work}
The most pressing issue is certainly to strengthen the parallelism between ITRM-randomness and ML-randomness by studying the corresponding notion for sets of Lebesgue measure $0$ rather than meager sets. Moreover, in focusing on
decidable sets as the relevant tests for ITRM-c-randomness, we deviate from the spirit of ML-randomness. An ITRM-based notion closer to the idea of ML-randomness would be `ITRM-ML-randomness': A set $X$ is an ITRM-ML-test
if and only if there is an ITRM-program $P$ such that (1) $P(i)$ computes (in an appropriate coding) a set $U_{i}$ of rational intervals with $\mu(\bigcup{U_{i}})\leq 2^{-i}$ for all $i\in\omega$ and (2) $X=\bigcap_{i\in\omega}\bigcup U_{i}$.
A real $x$ is then called ITRM-ML-random if it is not an element of any ITRM-ML-test. Note that, by solvability of the bounded halting problem for ITRMs, every ITRM-ML-test is ITRM-decidable, so that ITRM-ML-randomness is a weaker notion than
ITRM-randomness. We presently do not know, however, whether it is strictly weaker.

Still, ITRM-c-randomness shows an interesting behaviour, partly analogous to ML-randomness, though by quite different arguments. Similar approaches are likely to work for other machine models of generalized computations,
in particular ITTMs (\cite{HaLe}) (which were shown in \cite{CaSc} to obey the analogue of the non-meager part of Theorem \ref{ManyOracles}) and ordinal Turing Machines (\cite{Ko}) (for which the analogues of both parts of Theorem 
\ref{ManyOracles} turned out to be independent of ZFC) which we study in ongoing work. This further points towards a more general background theory of computation that allow unified arguments 
for all these various models as well as classical computability.
Furthermore, we want to see whether the remarkable conceptual stability of ML-randomness 
(for example the equivalence with Chaitin randomness or unpredictabiliy in the sense of r.e. Martingales, see e.g. sections $6.1$ and $6.3$ of \cite{DoHi})
carries over to the new context.

\section{Acknowledgements}

I am indebted to Philipp Schlicht for several helpful discussions of the results and proofs and suggesting various crucial references. I also thank the referees of \cite{CiE} for suggesting several simplifications and clarifications.
Finally, I thank the anonymous referee of this paper for various improvements of the presentation.


\begin{thebibliography}{}
 \bibitem{Ca} M. Carl. The distribution of ITRM-recognizable reals. In: S.B. Cooper, A. Dawar, M. Hyland and B. L\"owe (eds.). Turing Centenary Conference: How the World Computes. Annals of Pure and Applied Logic
Volume 165 (9), (2014), 1353-1532 

 \bibitem{Ca2} M. Carl. Optimal Results on ITRM-recognizability. To appear in the Journal of Symbolic Logic. Preprint available at arXiv:1306.5128v1 (2013)
 
\bibitem{CaSc} M. Carl, P. Schlicht. Infinite Computations with Random Oracles. To appear in the Notre Dame Journal of Formal Logic. Preprint available at arXiv:1307.0160 (2014)

\bibitem{CiE} M. Carl. Algorithmic Randomness for Infinite Time Register Machines. In: A. Beckmann, E. Csuhaj-Varju and Klaus Meer (eds.) Language, Life and Limits. 10th conference on Computability in Europe. Lecture Notes in Computer Science 8493 (2014), 84-92
 
\bibitem{Cu} N.J. Cutland. Computability - An introduction to recursive function theory. Cambridge University Press (1980)

 \bibitem{DoHi} R.G. Downey, D. Hirschfeldt. Algorithmic Randomness and  Complexity. Theory and Applications of Computability. Springer Berlin LLC (2010)

 \bibitem{HaLe} J. Hamkins, A. Lewis. Infinite Time Turing Machines. Journal of Symbolic Logic 65(2), (2000), 567-604

 \bibitem{HN} G. Hjorth, A. Nies. Randomness via effective descriptive set theory. Journal of the London Mathematical Society (2) 75 (2007), 495-508

 \bibitem{ITRM} P. Koepke, R. Miller. An enhanced theory of infinite time register machines. In: Logic and Theory of Algorithms. A. Beckmann et al, eds., Lecture Notes in Computer Science 5028 (2008), 306-315

 \bibitem{ITRM2} M. Carl, T. Fischbach, P. Koepke, R. Miller, M. Nasfi, G. Weckbecker. The basic theory of infinite time register machines. Archive for Mathematical Logic 49:249-273 No. 4, (2010) 850-856 

\bibitem{JeKa} R. Jensen, C. Karp. Primitive Recursive Set Functions. Axiomatic Set Theory, Proceedings of Symposia in Pure Mathematics, XIII, Part I, Providence, Rhode Island: American Mathematical Society,  (1971), 143-176
%

  \bibitem{Ka} A. Kanamori. The higher infinite. Large Cardinals in Set Theory from Their Beginnings. Springer Monographs in Mathematics. Springer (2005)

\bibitem{Mo} Y.N. Moschovakis. Descriptive Set Theory. Studies in Logic and the Foundations of Mathematics. vol. 100, North-Holland Publishing Company. New York (1980)

\bibitem{Ko} P. Koepke. Turing computations on ordinals. Bulletin of Symbolic Logic 11 (2005), 377-397

\bibitem{Ko2} P. Koepke. Ordinal Computability. Mathematical Theory and Computational Practice. K. Ambos-Spies et al, eds., Lecture Notes in Computer Science 5635 (2009), 280-289

 \bibitem{Ma} A.R.D. Mathias. Provident sets and rudimentary set forcing. To appear in
Fundamenta Mathematicae. Preprint available at \url{https://www.dpmms.cam.ac.uk/~ardm/fifofields3.pdf}


 \bibitem{Sa} G. Sacks. Higher recursion theory. Perspectives in Mathematical Logic, Volume 2. Springer Berlin (1990)
%
\bibitem{Si} S. Simpson. An extension of the recursively enumerable Turing degrees. Journal of the London Mathematical Society. Second Series, 75 (2007), 287-297, 
%
\bibitem{SW} S. Simpson, G. Weitkamp. High and Low Kleene Degrees of Coanalytic Sets. Journal of Symbolic Logic, 48(2) (1983), 356-368.
%
 \bibitem{So} R.I. Soare. Recursively Enumerable Sets and Degrees: A study of computable functions and computably generated sets.
 Perspectives in Mathematical Logic. Springer Heidelberg (1987)
%
 
\end{thebibliography}
\end{document}